\documentclass{article}

\usepackage{amsmath}
\numberwithin{equation}{section}
\usepackage{amssymb}
\usepackage{dsfont}
\usepackage{bm}
\usepackage{enumitem}
\usepackage{verbatim}
\usepackage{amsthm}

\usepackage[english]{babel}
\usepackage[utf8]{inputenc}

\theoremstyle{plain} 
\newtheorem{thm}{Theorem}[section]
\newtheorem{seur}[thm]{Corollary}
\newtheorem{lem}[thm]{Lemma}

\theoremstyle{remark}
\newtheorem{exa}[thm]{Example}
\newtheorem{rem}[thm]{Remark}

\newcommand{\X}{\textbf{X}}
\newcommand{\Sum}{\textbf{S}}
\newcommand{\Thetav}{\boldsymbol\Theta}
\newcommand{\Uv}{\textbf{U}}
\newcommand{\x}{\textbf{x}}
\newcommand{\y}{\textbf{y}}
\newcommand{\vv}{\textbf{v}}
\newcommand{\p}{\textbf{p}}
\newcommand{\e}{\textbf{e}}
\newcommand{\nolla}{\textbf{0}}

\newcommand{\RR}{\mathbb R}
\newcommand{\Ss}{\mathbb S}

\newcommand{\PP}{\mathbb P}
\newcommand{\EE}{\mathbb E}

\newcommand{\ind}{\mathds{1}}
\newcommand{\ud}{\mathrm{d}}

\usepackage[T1]{fontenc} 
\usepackage{setspace}

\title{Large deviations for a class of multivariate heavy-tailed risk processes used in insurance and finance}
\author{Miriam Hägele, Jaakko Lehtomaa}

\begin{document}
\maketitle

\subsection*{Abstract}
Modern risk modelling approaches deal with vectors of multiple components. The components could be, for example, returns of financial instruments or losses within an insurance portfolio concerning different lines of business. One of the main problems is to decide if there is any type of dependence between the components of the vector and, if so, what type of dependence structure should be used for accurate modelling. 

We study a class of heavy-tailed multivariate random vectors under a non-parametric shape constraint on the tail decay rate. This class contains, for instance, elliptical distributions whose tail is in the intermediate heavy-tailed regime, which includes Weibull and lognormal type tails. The study derives asymptotic approximations for tail events of random walks. Consequently, a full large deviations principle is obtained under, essentially, minimal assumptions. As an application, an optimisation method for a large class of Quota Share (QS) risk sharing schemes used in insurance and finance is obtained.

\subsection*{Keywords}
Large deviations, subexponential distribution, multivariate random walk, elliptical distribution

\subsection*{Classification}
60G50, 91B30

%-----------------------------------
%   Introduction
%-----------------------------------
\section{Introduction and assumptions}
\subsection{Introduction}
Applications in finance and insurance require multivariate models with heavy-tailed distributions to accurately describe multivariate risks. This includes understanding the possible dependence types of large observations. The case where such observations are restricted to a subset, say orthant, of the $d$-dimensional space $\mathbb{R}^d$ is studied in the setting of multivariate regular variation in \cite{Lehtomaa3}. Many studies on multivariate heavy-tailed distributions are built on the assumption of extremely heavy tails assuming  e.g.\ regular variation \cite{Hult3, Hult1, Mikosch1, Nyrhinen1}. In this paper, we concentrate on the less studied situation where the large observations can be found from any direction and where the tails are not as heavy as regularly varying tails. Such situations appear naturally in the case of financial returns of portfolios since the tails are often observed to have a lognormal type distribution \cite{hardy2001regime,jrfm11030052,Tegner} and the observations can be present in all orthants \cite{Lehtomaa3}.

We study asymptotic approximations of random walks, i.e.\ multivariate processes $(\Sum_n):=(\Sum_n)_{n=1}^\infty$ in $\mathbb{R}^d$ where $$\Sum_n= \X_1+\dots+\X_n$$
and the increments $\X,\X_1,\X_2,\ldots$ are independent and identically distributed (i.i.d.)\ random vectors. The class of studied increments is closely related to the class of multivariate subexponential vectors. Our class concerns lighter than polynomial tails where the variables have finite moments of all orders.
There exist at least three different approaches in the literature to define multivariate subexponentiality. The definitions in \cite{Cline1,Omey1} require, in addition to subexponentiality of the marginal distributions, a multivariate version of long-tailedness. The approach in \cite{Samorodnitsky} uses an alternative definition via fixed ruin sets in order to define a one-dimensional distribution function with respect to each set. The distribution class considered in this paper is consistent with the definition of \cite{Samorodnitsky}. For the one-dimensional case, \cite{Denisov1} provides an overview of large deviations results for subexponential distributions.

We write $\X$ in product form as
$$\X=R\Uv.$$ 
The one-dimensional radius variable $R$ controls the heaviness of the increments and $\Uv$ indicates which directions (defined by unit vectors) are possible. Variable $R$ can have, for example, Weibull or Lognormal type distribution. This definition can be extended to include the class of elliptical distributions, which frequently appear in the literature in applications in finance, see, for instance, \cite{Hult5,Kluppelberg2}. Notably, the tail decay speed of $R$ is not restricted to a narrowly defined parametric class. 

The proof methods are based on earlier results concerning one-dimensional random walks such as the ones presented in \cite{Lehtomaa2}. A full large deviations principle with non-trivial rate function under, essentially, minimal assumptions on the distribution is also derived. This study complements \cite{Mikosch3} which considers lognormal distributions and the result presented in \cite{Bazhba1} which focuses on Weibull distributions in the one-dimensional setting. As an application, we get an optimisation method for Quota Share (QS) risk sharing schemes, which are widely used in the field of reinsurance. In a QS-contract, there are two participants called the ceding company and the reinsurance company. They agree to share a random risk $Y$ so that one pays $qY$ and the other pays $(1-q)Y$. Our aim is to optimise the portions $q$ when a company buys reinsurance for all lines of business, i.e.\ each component of $\X$ is shared with a reinsurance company. The optimisation is obtained from the viewpoint of both the ceding and the reinsurance company.

%--------------------------------------------------------------------------
\subsection{Notation}
We denote vectors by bold symbols and their components by upper indices, e.g.\ for $\x\in \mathbb{R}^d$ we write $\x=(x^1,\dots,x^d)^T$. The inner product of the vectors $\x$ and $\y$ is denoted by $\langle \x,\y\rangle=\sum_{j=1}^d x^jy^j$ and $\|\x\|_2$ is the $L_2$-norm. Here, $\|\x\|_2$ is called the length of $\x$ and $\x/\|\x\|_2$ the direction of $\x$. 
$\Ss^{d-1}$ is the $d$-dimensional unit sphere, the subset of $\RR^d$ including all vectors with $L_2$-norm equal to one.
The notation $B^{\circ}$ stands for the interior of the set $B$, $\overline{B}$ for its closure and $B^c$ for its complement.
For $r\in\RR_+$ and $S\subset \Ss^{d-1}$, we set 
\begin{equation} \label{def_vrs}
V_{r,S}:=\left\{\x\in\RR^d:\|\x\|_2>r,\frac{\x}{\|\x\|_2}\in S\right\},
\end{equation}
where the expression $A:=B$ means $A$ is defined by $B$.
$B(\x,a)$ defines a ball centred at $\x$ with radius $a$ and we denote the inverse function of $f$ by $f^{-1}$. 

The asymptotic relation $f(x)\sim g(x)$, as $x\to\infty$ means $\lim_{x\to\infty} f(x)/g(x)=1$ and the little-o notation $g(x)=o(f(x))$ means $\lim_{x\to\infty} g(x)/f(x)=0$. We take the limit $x\to\infty$ or $n\to\infty$, where $x$ denotes real and $n$ natural numbers. The symbol $\ind(C)$ denotes the indicator function of the event $C$, $\PP(C)$ its probability and $\EE(X)$ stands for the expectation of $X$. By $A$ we denote a symmetric $d\times d$ matrix and $\Omega\subset \RR^d$ is the ellipsoid generated by the linear transformation $\Lambda:\RR^d\to\RR^d, \Lambda(\x)=A\x$ of the unit sphere, $\Omega=\Lambda(\Ss^{d-1})$.

%---------------------------------------------------------------------
\subsection{Model assumptions} \label{sec_as}
The aim is to derive a large deviations principle for elliptical multivariate distributions with moderate heavy tails. Therefore, we study the asymptotic behaviour of the random walk $(\Sum_n)$, where
\begin{displaymath}
\Sum_n = \X_1+\dots+\X_n,
\end{displaymath}
and $\X,\X_1,\X_2,\dots$ are i.i.d.\ increments. Here, $\X$ is the product of a heavy-tailed random variable $R$ and a random vector $\Uv$ or $\Thetav$ similarly to the setting in \cite{Hagele1}. We assume that $\Uv$ is distributed on the $d$-dimensional unit sphere $\Ss^{d-1}$ and that $\Thetav$ is distributed on a $d$-dimensional ellipsoid $\Omega$. 

We make the following, essentially minimal, assumptions on $R,\Uv$ and $\Thetav$.
\begin{enumerate}[label=(A\arabic*)]
\item \label{as_R}
The tail function of the random variable $R$ satisfies 
\begin{equation}\label{eq_as_eq1}
-\log(\PP(R>x)) \sim h(x),
\end{equation}
as $x\to\infty$, where $h(x)$ is an increasing and concave function such that 
\begin{enumerate}[label=(\roman*)]
    \item $h(x)=o(x)$ and
    \item $\log(x)=o(h(x))$, as $x\to\infty$.
\end{enumerate}  

\item \label{as_U}
The random vector $\Uv\in \Ss^{d-1}$ has a distribution on the $d$-dimensional unit sphere $\Ss^{d-1}$. Let $S\subset \Ss^{d-1}$ be a subset with positive Lebesgue measure. We assume that $\PP(\Uv\in S)>0$. In addition, $\Uv$ is assumed to be asymptotically independent of the random variable $R$ in the sense that
\begin{displaymath}
\lim_{x\to\infty} \PP(\Uv\in S | R>x) = \PP(\Uv\in S), 
\end{displaymath}
and $\EE(R\Uv)=\nolla$.
\end{enumerate}

We can then define the random vector $\Thetav$ by a linear transformation from the unit sphere $\Ss^{d-1}$ to the $d$-dimensional ellipsoid $\Omega$ centred at the origin. The linear transformation $\Lambda:\RR^d\to\RR^d$ with $\Lambda(\x)=A\x$, where $A$ is a symmetric, positive definite matrix generates the ellipsoid $\Omega=\Lambda(\Ss^{d-1})$. The random vector $\Thetav$ can then be written as a transformed vector, $\Thetav=\Lambda(\Uv)$. If $A$ is a diagonal matrix, the ellipsoid is orientated along the axes. 

Instead of defining $\Thetav$ through the linear transformation $\Lambda$, we can write its definition in a similar way as for the random vector $\Uv$.

\begin{enumerate}[label=(A2')]
    \item \label{as_thetav}
The $d$-dimensional random vector $\Thetav$ is distributed on an ellipse or ellipsoid $\Omega$ centred at the origin with $\EE(R\Thetav)=\nolla$. It holds for every set $S\subset \Omega$ with positive Lebesgue measure that $\PP(\Thetav\in S)>0$ and $\Thetav$ is asymptotically independent of the random variable $R$ in the sense that $\lim_{x\to\infty} \PP(\Thetav\in S | R>x) =  \PP(\Thetav\in S)$.
\end{enumerate}

\begin{rem}
 Assumption \ref{as_R} implies that the random variable $R$ is heavy-tailed in the sense that $\EE(e^{sR})=\infty$ for all $s>0$. The fact that $\log(x)=o(h(x))$ implies $\EE(R^s)<\infty$ for all $s>0$ so the random variable $R$ has finite moments of all orders. Furthermore, it follows from assumptions \ref{as_U} and \ref{as_thetav} that the support of the random vector $\Uv$ or $\Thetav$ is the entire set $\Ss^{d-1}$ or $\Omega$.
\end{rem}

Assumption \ref{as_R} is closely related to the class of subexponential distributions that is introduced, for instance, in \cite{Embrechts1, Foss1}.

\begin{lem}
If Assumption \ref{as_R} holds with \eqref{eq_as_eq1} as an equality for large enough arguments, the distribution of $R$ belongs to the class of subexponential distributions. 
\end{lem}
\begin{proof}
The statement follows directly from Theorem 2 of \cite{Teugels} which gives a sufficient condition for subexponentiality of a distribution based on tail functions. The condition has three requirements two of which are immediately true by our definition. To check the remaining condition, we can define an auxiliary function $g(x):=(x/h(x))^{1/2}$. Then $g(x)\to \infty$ and $x-g(x)\to\infty$, as $x\to\infty$. Without loss of generality, we can assume $h(0)\geq 0$, see Remark \ref{rem_hpilkku}. Due to the concavity, it holds that
\begin{eqnarray*}
&&\lim_{x\to\infty}\frac{\PP(R>x-g(x))}{\PP(R>x)}= \lim_{x\to\infty}\exp\left(-h\left(\left(1-\frac{g(x)}{x}\right)x\right)+h(x)\right)\\
&\geq& \lim_{x\to\infty}\exp\left(-\left(1-\frac{g(x)}{x}\right)h(x)+h(x)\right)
= \lim_{x\to\infty}\exp\left(\frac{g(x)}{x}h(x)\right)\\
&=&  \lim_{x\to\infty}\exp(g(x)^{-1})=1
\end{eqnarray*}
The corresponding upper bound of the limit is immediately valid by definition. 
\end{proof}

\begin{exa}
Simple examples of distributions of $R$ that fulfil Assumption \ref{as_R} include Weibull distributions with parameter $\beta\in (0,1)$ and lognormal type distributions which are defined by the relation $\PP(R>x) \sim e^{-(\log(x))^p}$ for $x>x_0$, where $p> 1$.
\end{exa}

Assumption \ref{as_R} can be used to obtain bounds even if it does not hold immediately for a given tail function $\PP(R>x)$. For example, if $R$ can be stochastically bounded by, say, $R'$ and $R''$ in the sense that 
$$\PP(R'>x)\leq \PP(R>x) \leq \PP(R''>x)$$
and the variables $R'$ and $R''$ satisfy \ref{as_R} (possibly with different concave functions), a result can be obtained if the asymptotic behaviour concerning the upper and lower bounds coincides in a suitable sense. For a concrete example of this, recall that a random variable $R$ belongs to the class of stretched exponential distributions if, for large enough $x$, inequalities
\begin{displaymath}
l_1(x)e^{-l(x)x^\beta}\leq \PP(R>x) \leq l_2(x)e^{-l(x)x^\beta}
\end{displaymath}
hold, where $\beta\in (0,1)$ and $l,l_1,l_2$ are slowly varying functions. This class is studied in particular in \cite{Gantert2,Gantert1}. Here, Assumption \ref{as_R} is valid if $l(x)x^\beta$ is a concave function for large enough $x$. If it is not, we can still find, based on Theorem 1 of \cite{Lehtomaa2}, a function $\underline{h}(x)$ which satisfies Assumption \ref{as_R} and inequality $\PP(R>x)\leq e^{-\underline{h}(x)}$ for large enough $x$ and
$$\liminf_{x\to \infty} \frac{-\log \PP(R>x)}{\underline{h}(x)}=1. $$
This fact can be used in the proofs by replacing $\PP(R>x)$ by $e^{-\underline{h}(x)}$ in suitable places in order to obtain results also for the stretched exponential class.

%---------------------------------------------------------------
%   Asymptotics of spherical distributions
%---------------------------------------------------------------
\section{Asymptotics of spherical distributions}\label{sec_sph}
\subsection{Large deviations principle}
Throughout this section, we study the random walk $(\Sum_n)$ generated by random vectors of the form $\X=R\Uv$, where the random variable $R$ fulfils Assumption \ref{as_R} and the random vector $\Uv$ fulfils Assumption \ref{as_U}.
We examine the probability of the asymptotic event that the random walk exceeds a threshold in a selected norm in order to prove a large deviations theorem. In this study, we choose to use the $L_2$-norm because it is, in our view, a natural choice when dealing with ellipses.

We start by considering a spherical distribution. The result is later extended to the setting of asymptotically elliptical heavy-tailed distributions. 
The proofs of the theorems stated below can be found in Subsection \ref{sec_proofsph}. The first result concerns logarithmic asymptotics of the norm of the random walk. 

\begin{thm} \label{thm_ball}
Let $a>0$ be a fixed number. Suppose the increment of the random walk $(\Sum_n)$ is of the form $\X=R\Uv$, where assumptions \ref{as_R} and \ref{as_U} hold. Then,
\begin{displaymath}
\lim_{n\to\infty} \frac{\log(\PP(\|\Sum_n\|_2>na))}{h(na)} = -1.
\end{displaymath}
\end{thm}

The asymptotic relation derived in Theorem \ref{thm_ball} yields a full large deviations principle with non-trivial rate function for asymptotically spherical heavy-tailed distributions under an additional technical assumption. 

\begin{thm} \label{thm_ldp}
Let $\X=R\Uv$ where $R$ and $\Uv$ fulfil assumptions \ref{as_R} and \ref{as_U}. Additionally, assume that, for $a>0$, the limit
\begin{equation}\label{eq_limh}
\lim_{x\to\infty} \frac{h(ax)}{h(x)}
\end{equation}
exists.
Then, the process $\{\Sum_n/n\}$ satisfies the large deviations principle with rate function 
\begin{displaymath}
I(\x) = \left\{\begin{array}{ll}\lim_{n\to\infty}\frac{h(n\|\x\|_2)}{h(n)}, & \textrm{if }\x\neq \nolla\\
0,& \textrm{if }\x=\nolla
\end{array}\right.
\end{displaymath}
and scale $h$, i.e.\
\begin{align*}
-\inf_{\y\in B^\circ} I(\y) &\leq \liminf_{n\to\infty} \frac{\log(\PP(\Sum_n/n\in B))}{h(n)}\\
&\leq \limsup_{n\to\infty} \frac{\log(\PP(\Sum_n/n\in B))}{h(n)} \leq -\inf_{\y\in \overline{B}} I(\y)
\end{align*}
for all Borel sets $B\subset \RR^d$.
\end{thm}

The large deviations principle in Theorem \ref{thm_ldp} is a multivariate equivalent of the large deviation principle in \cite{Lehtomaa2} for $d$-dimensional spherical random vectors.

\begin{seur} \label{cor_cont}
If, in addition to the assumptions of Theorem \ref{thm_ldp}, $I(\x)$ is continuous for all $\x\in \RR^d$ and the Borel set $B$ fulfils $\overline{B^\circ}=\overline{B}$, it holds that
\begin{displaymath}
\lim_{n\to\infty} \frac{\log(\PP(\Sum_n\in nB))}{h(n)}=-\inf_{\x\in B} I(\x).
\end{displaymath}
\end{seur}
\begin{proof}
The claim follows directly from the assumed continuity of the rate function.
\end{proof}

\begin{rem}\label{rem_h}
The existence of Limit (\ref{eq_limh}) implies
\begin{displaymath}
\lim_{x\to\infty}\frac{h(ax)}{h(x)}=a^\alpha
\end{displaymath}
for some $\alpha\geq 0$ due to Theorem 1.4.1 in \cite{Bingham} so $h(x)$ is, in fact, a regularly varying function.
\end{rem}

The rate function is symmetric with respect to the origin. Furthermore, the one-dimensional rate function along any line segment with endpoint in the origin is concave like the rate function in the one-dimensional case examined in \cite{Lehtomaa2}.

The following example examines the rate function for typical distributions that fulfil Assumption \ref{as_R}.
\begin{exa} \label{ex_r}
\begin{enumerate}[label=(\roman*)]
\item Let $R$ be Weibull distributed with parameter $\beta\in (0,1)$. Then, $h(x)=cx^\beta$ with some constant $c>0$ so the index $\alpha$ from Remark \ref{rem_h} is equal to $\beta$. Furthermore, 
\begin{displaymath}
I(\x) = \|\x\|_2^\beta
\end{displaymath}
for all $\x\in\RR^d$ and $I$ is a good rate function. The rate function $I(\x)$ is continuous so Corollary \ref{cor_cont} holds.
\item If $R$ has a lognormal type distribution of the form $\PP(R>x)\sim e^{-(\log(x))^p}$ for $x>x_0$ with some parameter $p>1$, the index $\alpha$ from Remark \ref{rem_h} is $0$ and 
\begin{displaymath}
I(\x)= 1
\end{displaymath}
for all $\x\in \RR^d\backslash \{\nolla\}$ so the rate function $I$ jumps at the origin. The rate function is non-trivial, but not good, according to the terminology used in the context of large deviations.
\end{enumerate}
\end{exa}

%--------------------------------------------------------------------------
\subsection{Auxiliary results}
In order to prove Theorem \ref{thm_ball} and Theorem \ref{thm_ldp}, we need some auxiliary lemmas. The auxiliary results study the projection of the random walk to a one-dimensional setting and its asymptotics.  

The orthogonal projection $P_\vv(\cdot)$ on the subspace spanned by the vector $\vv\in \Ss^{d-1}$ is defined as
\begin{displaymath}
P_\vv(\x) := \langle\vv,\x\rangle \vv,
\end{displaymath}
where the inner product defined as $p_\vv(\x):=\langle\vv,\x\rangle$ indicates the length of the projected vector in the subspace according to the $L_2$-norm.

The next result shows that the projection of the random vector $\X$ has the same asymptotic behaviour as $\|\X\|_2$.
\begin{lem} \label{lem_projh}
Suppose $\X=R\Uv$ where assumptions \ref{as_R} and \ref{as_U} hold. Let $a>0$ and $\vv\in\Ss^{d-1}$. Then
\begin{displaymath}
\lim_{n\to\infty} \frac{\log(\PP(p_\vv(\X)>na))}{h(na)} = -1.
\end{displaymath}
\end{lem}
\begin{proof}
The asymptotic upper bound is due to 
\begin{displaymath}
\log(\PP(p_{\vv}(\X)>na)) \leq \log(\PP(R>na))\sim -h(na)
\end{displaymath}
since $\langle \vv,\Uv\rangle \leq 1$.

To prove the asymptotic lower bound, we fix $\delta>0$ and set $S(\vv,\delta)\subset \Ss^{d-1}$ to be the $\delta$-environment of the vector $\vv$ in $\Ss^{d-1}$. Defining 
\begin{displaymath}
c_{\vv_\delta} := \min_{\y\in S(\vv,\delta)} \langle\vv,\y\rangle
\end{displaymath}
which is positive choosing $\delta$ small enough, it holds that
\begin{eqnarray*}
&&\log(\PP(p_{\vv}(\X)>na))\\
&\geq&  \log\left(\PP\left(\langle\vv,\Uv\rangle R>na,\Uv\in S(\vv,\delta)\right)\right) \\
&\geq& \log\left(\PP\left(R>\frac{na}{c_{\vv_\delta}}, \Uv\in S(\vv,\delta)\right)\right)\\
&=& \log\left(\PP\left(R>\frac{na}{c_{\vv_\delta}}\right)\right)+\log\left(\PP\left( \Uv\in S(\vv,\delta) \Big| R>\frac{na}{c_{\vv_\delta}}\right)\right).
\end{eqnarray*}
Hence,
\begin{eqnarray*}
\liminf_{n\to\infty} \frac{\log(\PP(p_\vv(\X)>na))}{h(na)} &\geq &
\liminf_{n\to\infty} \frac{\log\left(\PP(\Uv\in S(\vv,\delta))\right)-h\left(\frac{na}{c_{\vv_\delta}}\right)}{h(na)}\\
&=& \liminf_{n\to\infty} \frac{-h\left(\frac{na}{c_{\vv_\delta}}\right)}{h(na)}\\
\end{eqnarray*}
which holds for every $\delta>0$ small enough and proves thus the claim since $c_{\vv_\delta} \to 1$, as $\delta\to 0$.
\end{proof}

In the proof of Lemma \ref{lem_projbigjump} we divide the probability into the term caused by a single big jump and its complement. The following lemma provides an upper bound for the remaining term not caused by a single big jump.

\begin{lem} \label{lem_remterm}
Let $Y, Y_1, Y_2, \dots$ be i.i.d.\ real-valued random variables with $\EE(Y)=0$ and finite moments of all order and suppose $a>0$. Furthermore, let 
\begin{displaymath}
\liminf_{x\to\infty}\frac{-(\log(\PP(Y>x))}{h(x)}\geq 1,
\end{displaymath}
where $h(x)$ fulfils the assumptions on the function $h(x)$ stated in Assumption \ref{as_R} and  additionally $h(0)\geq 0$. Then, 
\begin{displaymath}
\limsup_{n\to\infty}\frac{\log\left(\PP\left(\sum_{i=1}^n Y_i>na, \max_{i=1}^n Y_i\leq na\right)\right)}{h(na)} \leq -1.
\end{displaymath}
\end{lem}

The proof uses similar ideas as the proof of Theorem 2 in \cite{Lehtomaa2}.

\begin{proof}[Proof of Lemma \ref{lem_remterm}:]
Due to the inequality
\begin{eqnarray*}
&&\EE\left(e^{b_n\sum_{i=1}^n Y_i}\ind\left(\sum_{i=1}^n Y_i>na, \max_{i=1}^n Y_i \leq na\right)\right)\\
&\geq& e^{b_n na}\PP\left(\sum_{i=1}^n Y_i>na, \max_{i=1}^n Y_i \leq na\right)
\end{eqnarray*}
and the independence of the random variables, it holds that
\begin{displaymath}
\PP\left(\sum_{i=1}^nY_i>na, \max_{i=1}^n Y_i\leq na\right) \leq \exp\left(-b_n na\right)\left(\EE\left(e^{b_nY}\ind(Y\leq na)\right)\right)^n.
\end{displaymath}
To bound the expectation from above, we split it into two parts
\begin{eqnarray*}
\EE\left(e^{b_nY}\ind(Y\leq na)\right) &=&
\EE\left(e^{b_nY}\ind(Y\leq c_n)\right) +\EE\left(e^{b_nY}\ind(c_n<Y\leq na)\right) \\
&=& E_1+E_2.
\end{eqnarray*}
Taking $\delta\in (0,1)$ and setting $b_n:=(1-\delta)h(na)/na \to 0$, one can choose $\varepsilon(n)\to 0$ such that $c_n:= b_n^{-1}\varepsilon(n)\to \infty$, as $n\to\infty$. Then, it holds $b_n c_n \to 0$ and therefore one can apply Taylor series to the first term $E_1$,
\begin{eqnarray*}
E_1 &=& \EE((1+b_nY+o(b_nY))\ind(Y\leq c_n))\\
&=& \PP(Y\leq c_n) + b_n \EE(Y\ind(Y\leq c_n))(1+o(1)). 
\end{eqnarray*}
Integrating $E_2$ by parts and rewriting it in terms of the tail distribution of $Y$, one gets
\begin{eqnarray*}
E_2 &=& e^{b_n na}\PP(Y\leq na)-e^{b_n c_n}\PP(Y\leq c_n)-\int_{c_n}^{na} b_n e^{b_n x}\PP(Y\leq x)\ud x\\
&=&e^{b_n c_n} \PP(Y>c_n) -e^{b_n na}\PP(Y>na) + b_n\int_{c_n}^{na} e^{b_n y} \PP(Y>y)\ud y.
\end{eqnarray*}
For every $\delta>0$, it holds $\PP(Y>y)\leq\exp(-(1-\delta/2)h(y))$ for all $y\geq c_n$ choosing $n$ large enough.
Applying additionally Taylor series to the first term of the latter equation results in the upper bound
\begin{eqnarray*}
E_2 &\leq& (1+b_n c_n+o(b_n c_n))\PP(Y>c_n) + b_n \int_{c_n}^{na}\exp\left(b_n y-\left(1-\frac{\delta}{2}\right) h(y)\right) \ud y.
\end{eqnarray*}
Due to the concavity of the function $h(x)$ and the fact that $y\leq na$, it holds that
\begin{displaymath}
h(y)\geq \frac{y}{na}h(na)
\end{displaymath}
and therefore
\begin{eqnarray*}
&& b_n \int_{c_n}^{na}\exp\left(b_n y-\left(1-\frac{\delta}{2}\right)h(y)\right) \ud y \\
&=& b_n \int_{c_n}^{na}\exp\left((1-\delta)\frac{y}{na}h(na)-(1-\delta)h(y)-\frac{\delta}{2} h(y)\right) \ud y \\
&\leq& b_n \int_{c_n}^{na}\exp\left(-\frac{\delta}{2} h(y)\right) \ud y
\leq b_n na\exp\left(-\frac{\delta}{2} h(c_n)\right).
\end{eqnarray*}
Applying the inequality $\log(x)\leq x-1$ to the term $n\log(E_1+E_2)/h(na)$ yields
\begin{eqnarray*}
&&\frac{n}{h(na)}\log\left(\EE\left(e^{b_n Y}\ind(Y\leq na)\right)\right) \\
&\leq& \frac{n}{h(na)}\log\bigg(\PP(Y\leq c_n) + b_n \EE(Y\ind(Y\leq c_n))(1+o(1))+\PP(Y>c_n)\\ 
&& + \ b_n c_n\PP(Y>c_n)(1+o(1))+b_n na\exp\left(-\frac{\delta}{2} h(c_n)\right)\bigg)\\
&\leq&  \frac{1-\delta+o(1)}{a}\left(\EE(Y\ind(Y\leq c_n))+c_n\PP(Y>c_n)\right)+(1-\delta)n\exp\left(-\frac{\delta}{2} h(c_n)\right).
\end{eqnarray*}
The first terms converge to zero, because $\EE(Y)=0$ and all moments are finite. Choosing $\varepsilon(n)$ such that $\log(n)=o\left(h(c_n)\right)$, also the last term converges to zero.
Finally,
\begin{eqnarray*}
&&\limsup_{n\to\infty}\frac{\log\left(\PP\left(\sum_{i=1}^n Y_i>na, \max_{i=1}^n Y_i\leq na\right)\right)}{h(na)}\\
&\leq& \limsup_{n\to\infty}\frac{-b_n na}{h(na)} + \frac{n}{h(na)}\log\left(\EE\left(e^{b_n Y}\ind(Y\leq na)\right)\right) =-(1-\delta).
\end{eqnarray*}
The last inequality holds for any $\delta>0$ which implies the claim.
\end{proof}

\begin{rem} \label{rem_hpilkku}
If $h$ is an increasing concave function with $h(x)=o(x)$ and $\log(x)=o(h(x))$, one can always construct an increasing concave function $h'$ with $h'(0)\geq 0$ and $h'(x)\sim h(x)$, as $x\to \infty$ by changing $h$ to a suitable linear function near the origin. This means that $h'$ is a subadditive function.
\end{rem}

We can now state the principle of a single big jump for projections of multivariate random walks. 

\begin{lem} \label{lem_projbigjump}
Let $\X=R\Uv$, assume \ref{as_R} and \ref{as_U} and let $\vv\in \Ss^{d-1}$ and $a>0$ fixed. Then, it holds that
\begin{displaymath}
\lim_{n\to\infty} \frac{\log(\PP(p_\vv(\Sum_n)>na))}{\log(\PP(p_\vv(\X)>na))}=1.
\end{displaymath}
\end{lem}
\begin{proof}
At first, we show the asymptotic lower bound
\begin{displaymath}
\liminf_{n\to\infty} -\frac{\log(\PP(p_\vv(\Sum_n)>na))}{\log(\PP(p_\vv(\X)>na))}\geq -1.
\end{displaymath}
Using the principle of a single big jump and the weak law of large numbers, it follows that
\begin{eqnarray*}
&&\liminf_{n\to\infty} -\frac{\log(\PP(p_\vv(\Sum_n)>na))}{\log(\PP(p_\vv(\X)>na))}\\
&\geq&\liminf_{n\to\infty} -\frac{\log(\PP(p_\vv(\Sum_{n-1})\leq \varepsilon na,p_\vv(\X_n)>(1+\varepsilon)na))}{\log(\PP(p_\vv(\X)>na))}\\
&\geq&\liminf_{n\to\infty} -\frac{\log(\PP(p_\vv(\Sum_{n-1})\leq \varepsilon na))}{\log(\PP(p_\vv(\X)>na))}-\frac{\log(\PP(p_\vv(\X)> (1+\varepsilon)na))}{\log(\PP(p_\vv(\X)>na))}\\
&=& \liminf_{n\to\infty} -\frac{h((1+\varepsilon)na)}{h(na)}.
\end{eqnarray*}
The last expression applies, additionally to the weak law of large numbers, Lemma \ref{lem_projh}. The fact that this holds for every $\varepsilon>0$ implies the asymptotic lower bound.

It remains to show the asymptotic upper bound
\begin{displaymath}
\limsup_{n\to\infty} -\frac{\log(\PP(p_\vv(\Sum_n)>na))}{\log(\PP(p_\vv(\X)>na))}\leq -1.
\end{displaymath}
Dividing the probability into the case where at least the projection of one random vector exceeds the threshold $na$ and its complement implies
\begin{eqnarray*}
&&\PP(p_\vv(\Sum_n)>na)\\
&\leq& n\PP(p_\vv(\X)>na)+\PP(p_\vv(\Sum_n)>na,p_\vv(\X_i)\leq na \textrm{ for all }i=1,\dots n).
\end{eqnarray*}
Applying Lemma 1.2.15 in \cite{Dembo1},  
\begin{displaymath}
\limsup_{\varepsilon\to 0} \varepsilon\log(a_\varepsilon^1+a_\varepsilon^n) = \max\left( \limsup_{\varepsilon\to 0} \varepsilon a_\varepsilon^1,\limsup_{\varepsilon\to 0} \varepsilon a_\varepsilon^2\right),
\end{displaymath}
we can examine the terms separately since $\left(\log(\PP(p_\vv(\X)>na))\right)^{-1} \to 0$, as $n\to\infty$. We get
\begin{eqnarray*}
&&\limsup_{n\to\infty} -\frac{\log(\PP(p_\vv(\Sum_n)>na))}{\log(\PP(p_\vv(\X)>na))}\\
&\leq& \max\left(\limsup_{n\to\infty} -\frac{\log(n\PP(p_\vv(\X)>na))}{\log(\PP(p_\vv(\X)>na))},\right.\\
&&\left. \limsup_{n\to\infty} -\frac{\log(\PP(p_\vv(\Sum_n)>na,p_\vv(\X_i)\leq na \textrm{ for all }i=1,\dots n))}{\log(\PP(p_\vv(\X)>na))}\right).
\end{eqnarray*}
Clearly, the first term is $-1$. For the second term, we set $Y=p_\vv(\X)$. Since
\begin{displaymath}
\log(\PP(p_\vv(\X)>na))=\log(\PP(\langle \vv,\Uv \rangle R>na)) \leq \log(\PP(R>na)) \sim -h(na),
\end{displaymath}
we can apply Lemma \ref{lem_remterm} with the help of Remark \ref{rem_hpilkku}. Finally, applying Lemma \ref{lem_projh} we get the upper bound $-1$ also for the second term which completes the proof.
\end{proof}

%------------------------------------------------------------------
\subsection{Proof of Theorem \ref{thm_ball} and Theorem \ref{thm_ldp}} \label{sec_proofsph}
We can now state the proofs of the main results of Section \ref{sec_sph}. 

% Proof thm_ball
\begin{proof}[{\bf Proof of Theorem \ref{thm_ball}:}]
First, we approximate the event $\{\|\Sum_n\|_2>na\}$ by projections in different directions. The fact that the principle of a single big jump holds for any orthogonal projection yields the desired asymptotic behaviour.
Since $\PP(\|\Sum_n\|_2>na)\geq \PP(p_\vv(\Sum_n)>na)$ for every $\vv\in \Ss^{d-1}$ the asymptotic lower bound 
\begin{displaymath}
\liminf_{n\to\infty} \frac{\log(\PP(\|\Sum_n\|_2>na))}{h(na)} \geq -1
\end{displaymath}
is an immediate consequence of Lemma \ref{lem_projh} and Lemma \ref{lem_projbigjump}.

To prove the corresponding upper bound, we cover the set $\{\x:\|\x\|_2>na\}$ by a finite union of $m$ sets defined by orthogonal projections and study the limit, as $m\to \infty$.
To this end, let $m\geq 2d$. We aim to choose vectors $\vv_k\in\Ss^{d-1}, k=1,\dots,m$, such that unions of form $\cup_{k=1}^m\{\x:p_{\vv_k}(\x)>\varepsilon\}$, where $\varepsilon>0$, can be used to cover the whole space except some neighbourhood of the origin. For example, in $\RR^2$, an easy way to define the vectors $\vv_k\in \Ss^{1}$ is to take $\vv_k= \left(\cos(2k\pi/m),\sin(2k\pi/m)\right)^T$. 

Choosing $\vv_k$ appropriately, for instance, such that they are uniformly spaced on the unit sphere, there exists some $\varepsilon(m)>0$ such that 
\begin{displaymath}
D:=\cup_{k=1}^m \{\x:p_{\vv_k}(\x)> (1-\varepsilon(m))na\} \supset \{\x: \|\x\|_2 >na\}
\end{displaymath}
and, more specifically, we can choose the numbers $\varepsilon(m)$ so that $\varepsilon(m)\to 0$, as $m\to\infty$.
Hence,
\begin{displaymath}
\PP(\|\Sum_n\|_2>na) \leq \PP(\Sum_n \in D) \leq \sum_{k=1}^m \PP(p_{\vv_k}(\Sum_n)>(1-\varepsilon(m))na).
\end{displaymath}
Lemma 1.2.15 in \cite{Dembo1} yields 
\begin{eqnarray*}
&&\limsup_{n\to\infty}\frac{\log(\PP(\|\Sum_n\|_2>na))}{h(na)}\\
&\leq& \limsup_{n\to\infty}\frac{\log\left(\sum_{k=1}^{m} \PP(p_{\vv_k}(\Sum_n)>(1-\varepsilon(m))na)\right)}{h(na)}\\
&\leq& \max_{k=1}^{m}\left(\limsup_{n\to\infty}\frac{\log\left(\PP(p_{\vv_k}(\Sum_n)>(1-\varepsilon(m))na)\right)}{h(na)}\right).
\end{eqnarray*}

Applying Lemma \ref{lem_projbigjump} and Lemma \ref{lem_projh}, we get
\begin{eqnarray*}
\limsup_{n\to\infty}\frac{\log(\PP(\|\Sum_n\|_2>na))}{h(na)}
&\leq& \limsup_{n\to\infty}\frac{-h((1-\varepsilon(m))na)}{h(na)}.
\end{eqnarray*}
Letting $m\to\infty$, it holds that $\varepsilon(m)\to 0$ and the term above converges to $-1$.
\end{proof}

% proof thm_ldp
\begin{proof}[{\bf Proof of Theorem \ref{thm_ldp}:}]
The proof of the large deviations principle is based on the fact that 
\begin{equation} \label{prob_VnaS}
\lim_{n\to\infty} \frac{\log(\PP(\Sum_n\in V_{na, S}))}{h(na)}=-1,
\end{equation}
since open sets of the form $V_{a,S}$ defined in (\ref{def_vrs}), where $a>0$ and $S$ is an open subset of $\Ss^{d-1}$, generate $\RR^d$. 

The limit superior of (\ref{prob_VnaS}) follows directly from Theorem \ref{thm_ball} due to the inequality
\begin{displaymath}
\PP(\Sum_n\in V_{na,S}) \leq \PP(\|\Sum_n\|_2>na).
\end{displaymath}
Rewriting
\begin{displaymath}
\PP(\Sum_n\in V_{na,S}) = \PP(\|\Sum_n\|_2>na)\PP\left(\frac{\Sum_n}{\|\Sum_n\|_2}\in S\  \Big|\ \|\Sum_n\|_2>na\right) 
\end{displaymath}
yields the limit inferior of (\ref{prob_VnaS}) since the last probability is positive.

The weak law of large numbers implies $\lim_{n\to\infty} \PP(\Sum_n/n \in B(\nolla,\varepsilon))=1$ for all $\varepsilon>0$ and thus if $\nolla\in B\subset \RR^d$
\begin{displaymath}
\lim_{n\to\infty}\frac{\log(\PP(\Sum_n/n \in B))}{h(n)} = 0.
\end{displaymath}
Finally, by (\ref{prob_VnaS}), it is easy to obtain the inequalities
\begin{align*}
\limsup_{n\to\infty} \frac{\log(\PP(\Sum_n/n\in F))}{h(n)}&\leq -\inf_{y\in F} I(y) \textrm{ for all closed sets }F\subset \RR^d,\\
\liminf_{n\to\infty} \frac{\log(\PP(\Sum_n/n\in G))}{h(n)}&\geq -\inf_{y\in G} I(y) \textrm{ for all open sets }G\subset \RR^d
\end{align*}
which result in the full large deviations principle.
\end{proof}

%-----------------------------------------------------------
%    Asymptotics of elliptical distributions
%-----------------------------------------------------------
\section{Asymptotics in the elliptical case}

\subsection{Contraction principle}
To extend the result to asymptotically elliptically distributed random vectors, we apply the contraction principle, Theorem 4.2.1 in \cite{Dembo1}, to the large deviations result in Theorem \ref{thm_ldp}.

In this section, we study the asymptotics of the random walk $(\Sum_n)$ with increments $\X=R\Thetav$ where $\Thetav$ is distributed on some $d$-dimensional ellipsoid $\Omega$ centred at the origin.
For asymptotically elliptical distributions a suitable linear map in the contraction principle is the bijective function $\Lambda:\RR^d\to\RR^d$ defined by $\Lambda(\x)=A\x$ mapping vectors from $\Ss^{d-1}$ to $\Omega$. Here, $A$ is a symmetric, positive definite and thus invertible $d\times d$ matrix such that $\Lambda(\Ss^{d-1})=\Omega$. Since $\Lambda$ is linear, 
$$\Lambda(\Sum_n) = \sum_{i=1}^n \Lambda(\X_i)$$ 
holds. 

\begin{thm} \label{thm_ldpellipse}
Let $\X=R\Thetav$ where $R$ and $\Thetav$ fulfil assumptions \ref{as_R} and \ref{as_thetav} and Limit (\ref{eq_limh}) exists. 
Let $A$ be a symmetric, positive definite $d\times d$ matrix and define $\Lambda:\RR^d\to\RR^d$ by $\Lambda(\x)=A\x$ such that $\Lambda$ maps $\Ss^{d-1}$ to $\Omega$. Then, the process $\{\Sum_n/n\}$ satisfies the large deviations principle with scale $h$, so for all Borel sets $B\subset \RR^d$
\begin{align*}
-\inf_{\y\in \Lambda^{-1}(B^\circ)} I(\y) &\leq \liminf_{n\to\infty} \frac{\log(\PP(\Sum_n/n\in B))}{h(n)}\\
&\leq \limsup_{n\to\infty} \frac{\log(\PP(\Sum_n/n\in B))}{h(n)} \leq -\inf_{\y\in \Lambda^{-1}(\overline{B})} I(\y)
\end{align*}
where  
\begin{displaymath}
I(\x) = \left\{\begin{array}{ll}\lim_{n\to\infty}\frac{h(n\|\x\|_2)}{h(n)}, & \textrm{if }\x\neq \nolla\\
0,& \textrm{if }\x=\nolla
\end{array}\right.
\end{displaymath}
and $\Lambda^{-1}:\RR^d\to\RR^d$ with $\Lambda^{-1}(\x)=A^{-1}\x$.
\end{thm}

\begin{exa}
Similarly to Example \ref{ex_r}, if $\X$ fulfils the assumptions of Theorem \ref{thm_ldpellipse} where $R$ follows a Weibull distribution with parameter $\beta\in(0,1)$, 
\begin{displaymath}
I(\x)= \|\x\|^\beta_2
\end{displaymath}
is a good rate function. If $R$ has a lognormal type distribution, the rate function is a constant everywhere except at the origin and hence not good.
\end{exa}

%-----------------------------------------------------------------------
\subsection{Proof of Theorem \ref{thm_ldpellipse}}
The proof of the main result in this section relies on the linearity of the function $\Lambda$ and the contraction principle for large deviation principles. 

% Proof of thm_ldpellipse
\begin{proof}[{\bf Proof of Theorem \ref{thm_ldpellipse}:}] 
The linear transformation
\begin{displaymath}
\Lambda(\x) = A\x
\end{displaymath}
where $A$ is a symmetric, positive definite $d\times d$ matrix is a bijective function $\Lambda:\RR^d\to\RR^d$. Its inverse function $\Lambda^{-1}:\RR^d\to\RR^d$ exists due to the invertibility of the matrix $A$ and is $\Lambda^{-1}(\x)=A^{-1}\x$. The linear transformation $\Lambda$ maps $\Ss^{d-1}$ to $\Omega$. By the linearity of $\Lambda$, it holds $\Lambda(\Sum_n)=\sum_{i=1}^n\Lambda(\X_i)$ so the mapping of the random walk is equal to the sum of mappings of the increments. 

The claim follows then from the contraction principle, see, for instance, Theorem 4.2.1 in \cite{Dembo1}. 
 Applying the contraction principle with the continuous function $\Lambda^{-1}$ we get 
\begin{eqnarray*}
\inf_{\y\in B} J(\y)&=& \inf_{\y\in B}\inf_{\x\in \RR^d} \{I(\x): \Lambda(\x)=\y\}\\
&=&\inf_{\x\in \RR^d}\{I(\x):\x\in \Lambda^{-1}(B)\} =\inf_{\x\in \Lambda^{-1}(B)} I(\x),
\end{eqnarray*}
which completes the proof.
\end{proof}

%----------------------------------------------------------
%    Applications to reinsurance
%----------------------------------------------------------
\section{Applications to reinsurance}
\subsection{Introduction and setting}
An insurance company with $d$ lines of business might optimise the asymptotic behaviour of their ruin probability by sharing its risks for some lines of business using quota share reinsurance contracts. Quota share reinsurance is a proportional reinsurance \cite{Embrechts1}, where the insurer (the ceding company) pays only a fixed ratio for each claim while the reinsurance company pays the rest. In general contracts, the insurance company shares both the losses and the profits with the reinsurer.

A diagonal $d\times d$ matrix $Q$ can be used to represent a quota share reinsurance strategy for an insurance company with $d$ lines of business. The element $q_{k,k}$ refers then to the quota share ratio of the $k$th line of business, i.e.\ the ceding company pays $q_{k,k}X^k$ of the $k$th line of business and the reinsurance company pays the remaining part $(1-q_{k,k})X^k$. It is natural to assume that $q_{k,k}\in (0,1]$ for all indices $k=1,\dots , d$ since typically the insurance company keeps some share for every line of business. Under this assumption Matrix $Q$ is invertible.

For the ceding company, the aim could be to find a quota share reinsurance strategy defined by a matrix $Q$ such that 
\begin{equation}\label{eq_suitableq}
\lim_{n\to\infty} \frac{\log(\PP(\|\Sum_n/n\|_2>a))}{h(n)} > \lim_{n\to\infty} \frac{\log(\PP(\|Q\Sum_n/n\|_2>a))}{h(n)}
\end{equation}
because it reduces the asymptotic size of the ruin probability $\PP(\|\Sum_n/n\|_2>a)$. 

We look at the quota share reinsurance from two different perspectives. In Subsection \ref{sec_qsins}, we model the optimal quota share reinsurance strategy from the point of view of the ceding company which sets reinsurance only for the lines of business with highest risks. 
Subsection \ref{sec_qsre} compares different quota share risk sharing strategies from the viewpoint of a reinsurance company that wants to offer quota share reinsurance while minimising their own risks.

% assumptions 
We model the risk process of an insurance company with $d$ lines of business as a $d$-dimensional random walk $(\Sum_n)$. The component of the increment $\X$ represents the difference between the claim size and the associated net premium in the corresponding line of business. We assume that $\X$ is asymptotically elliptically distributed: 
Let $\X=R\Thetav$ where $R$ and $\Thetav$ fulfil assumptions \ref{as_R} and \ref{as_thetav} and $\Omega$ is an ellipse or ellipsoid defined by the bijection $\Lambda:\RR^d\to \RR^d$ with $\Lambda(\x)=A\x$ where $A$ is a symmetric, positive definite $d\times d$ matrix.  

If there are several reinsurance strategies that yield the same right-hand side of Inequality (\ref{eq_suitableq}), the insurance company chooses the one with the smallest premium. In this setting, we define the premium of the reinsurance strategy $Q$ as
\begin{equation}\label{def_premium}
p(Q)= {\bf 1}^T(I-Q)\p =\sum_{j=1}^d (1-q_{j,j})p_j
\end{equation}
where $\p$ is a positive premium vector, ${\bf 1}$ denotes the $d$-dimensional vector with ones and $I$ the $d\times d$ identity matrix. The positive premium vector $\p$ contains the premium rates of the reinsurances for each line of business when the entire component is reinsured. For example, $p_j$ could be connected to the expected loss of the reinsurance company added with a safety loading. 
In (\ref{def_premium}), the constants $p_j$ are multiplied by the factors $(1-q_{j,j})$, where the $q$-coefficients can be selected by the ceding company. The premium vector is considered as given and the ceding company cannot affect the values of this vector. Therefore, (\ref{def_premium}) is the premium for the entire reinsurance strategy. 
If the insurance company does not take reinsurance for the $j$th line of business, it selects $q_{j,j}=1$. Hence, the premium for the reinsurance of the $j$th line of business is zero in this case.

%--------------------------------------------------------------------
\subsection{Quota share reinsurance strategy of the ceding company}\label{sec_qsins}
The aim of the insurance company is to identify the riskiest lines of business and choose its quota share reinsurance strategy reducing these risks. The ceding company typically only wants to insure the highest risks. This is why it is natural to assume $q_{k,k}=1$ for at least one index $k$ which represents the line of business with the lowest asymptotic risk.

Since quota share reinsurance is defined component-by-component one can find an optimal reinsurance strategy for distributions on ellipsoids orientated along the axes. The general case can be mathematically reduced into this setting by rotating the original data. However, if the data is transformed using a rotation, suitable QS contracts might not be immediately available on the market because the new axes would not correspond to the lines of business. 

\begin{rem}
If $\Omega$ is a $d$-dimensional ellipsoid orientated along the axes, there exists a positive, diagonal $d\times d$ matrix $A$ such that the mapping $\Lambda:\RR^d\to\RR^d, \Lambda(\x)=A\x$ of the unit sphere $\Ss^{d-1}$ generates $\Omega$. Furthermore, if $A$ is a positive, diagonal $d\times d$ matrix, the ellipsoid generated by $\Lambda(\Ss^{d-1})$ is orientated along the axes. 
\end{rem}

\begin{thm}\label{thm_reins}
Let $\X=R\Thetav$ where $\Thetav$ fulfils assumption \ref{as_thetav} and $\Omega=\Lambda(\Ss^{d-1})$ is defined by the linear transformation $\Lambda(\x)=A\x$, where $A$ is a diagonal $d\times d$ matrix  with $a_{k,k}>0$ for all $k=1,\dots,d$. Furthermore, we assume that $R$ follows a Weibull distribution with parameter $\beta\in (0,1)$. 
Then, the quota share reinsurance strategy defined by the matrix $Q=\min_{j=1}^d a_{j,j} A^{-1}$ yields the inequality
\begin{equation} \label{eq_asrel}
\lim_{n\to\infty} \frac{\log(\PP(\|\Sum_n/n\|_2>a))}{h(n)} \geq \lim_{n\to\infty} \frac{\log(\PP(\|Q\Sum_n/n\|_2>a))}{h(n)} 
\end{equation}
for any $a>0$ and $Q$ minimises the right-hand side of Inequality (\ref{eq_asrel}) over all strategies. The minimum is unique under the additional condition that $p(Q)$ is also minimised. 
\end{thm}
\begin{proof}
The assumptions on $R$ imply the large deviations principle 
\begin{displaymath}
\lim_{n\to\infty} \frac{\log(\PP(\|Q\Sum_n/n\|_2>a))}{h(n)} =-\inf_{\y\in B(\nolla, a)^c} \|A^{-1}Q^{-1}\y\|_2^\beta.
\end{displaymath}
Therefore, it is sufficient to show that the matrix $Q$ is the matrix that maximises 
\begin{equation} \label{eq_maksimoitava}
\max_{\tilde{Q}\in \mathcal{Q}}\inf_{\x\in B(\nolla, a)^c} \|A^{-1}\tilde{Q}^{-1}\x\|_2
\end{equation}
where $\mathcal{Q}$ is the set of $d\times d$ diagonal matrices with $q_{j,j}\in (0,1]$ for all $j=1,\dots,d$ and $q_{k,k}=1$ for at least one index $k\in \{1,\dots,d\}$. 
Without loss of generality we assume $\min_{j=1}^d a_{j,j}= a_{1,1}$. By the property $\|c\x\|_2=|c|\|\x\|_2$, the infimum is always achieved at the boundary, so
\begin{displaymath}
\inf_{\x\in B(\nolla, a)^c} \|A^{-1}\tilde{Q}^{-1}\x\|_2 = \inf_{\|\x\|_2=a}\|A^{-1}\tilde{Q}^{-1}\x\|_2.
\end{displaymath}
The fact that $\inf_{\|\x\|_2=a}\|A^{-1}Q^{-1}\x\|_2= a/a_{1,1}$ follows directly from the definition of $Q$ since
\begin{displaymath}
A^{-1}Q^{-1}\x= A^{-1}(\min_{j=1}^d a_{j,j} A^{-1})^{-1}\x = \frac{1}{a_{1,1}}A^{-1}A\x=\frac{1}{a_{1,1}}\x.
\end{displaymath}
It remains to show that a matrix $\tilde{Q}$ that maximises (\ref{eq_maksimoitava}) is of the form $\tilde{q}_{k,k} \leq a_{1,1}/a_{k,k}$ for all $k=1,\dots,d$. In order for $\tilde{Q}$ to maximise (\ref{eq_maksimoitava}) it has to hold that
\begin{displaymath}
\|A^{-1}\tilde{Q}^{-1}\x\|_2 = \left(\sum_{i=1}^d\frac{x_i^2}{a_{i,i}^2\tilde{q}_{i,i}^2}\right)^\frac{1}{2} \geq \frac{a}{a_{1,1}}
\end{displaymath}
for all $\x$ with $\|\x\|_2=a$. Checking the inequality for $a$ times the unit vectors we get the condition $\tilde{q}_{k,k} \leq a_{1,1}/a_{k,k}$ for all $k=1,\dots,d$. Due to the additional condition that $q_{k,k}=1$ for at least one index, we need to set $q_{1,1}=1$. Now, taking $\tilde{q}_{1,1}=1$ and $\tilde{q}_{k,k}<a_{1,1}/a_{k,k}$ for all $k=2,\dots,d$ yields
\begin{displaymath}
\inf_{\|\x\|_2=a} \|A^{-1}\tilde{Q}^{-1}\x\|_2 = \left(\sum_{i=1}^d\frac{x_i^2}{a_{i,i}^2\tilde{q}_{i,i}^2}\right)^\frac{1}{2} \leq \left(\frac{a^2}{a_{1,1}^2\tilde{q}_{1,1}^2}\right)^\frac{1}{2}= \frac{a}{a_{1,1}}
\end{displaymath}
so 
\begin{displaymath}
\inf_{\|\x\|_2=a}\|A^{-1}\tilde{Q}^{-1}\x\|_2 = \inf_{\|\x\|_2=a}\|A^{-1}Q^{-1}\x\|_2.
\end{displaymath}
Comparing the premium of the reinsurance strategies $Q$ and $\tilde{Q}$, it is easy to obtain $p(Q)<p(\tilde{Q})$ due to the positivity of the vector $\p$. Thus, $Q$ is the quota share reinsurance strategy that maximises (\ref{eq_maksimoitava}) with the lowest premium.
\end{proof}

\begin{rem}
Theorem \ref{thm_reins} can be extended to distributions for which Assumption \ref{as_R} hold and, in addition, $\lim_{n\to\infty} h(n\|\x\|_2)/h(n)=a^\alpha$ for $\alpha>0$. If $\lim_{n\to\infty} h(n\|\x\|_2)/h(n)$ is a constant, a quota share reinsurance does not improve the asymptotic behaviour of the logarithmic ruin probability and hence Theorem \ref{thm_reins} does not hold for lognormal type distributions. 
\end{rem}

The optimal matrix $Q$ transforms the set $\Omega$ to a $d$-dimensional ball. Hence, the probability that the reinsured risk process exceeds a threshold in a selected norm has the same asymptotics in all directions. Taking reinsurance defined by the matrix $Q$ reduces the risks of the riskier lines of business to the same level of the less risky line of business. 

\begin{rem}
If $a_{k,k}= a_{1,1}$ for all $k\in \{1,\dots,d\}$, $\Omega$ describes a $d$-dimensional ball and quota share reinsurance does not improve the asymptotic behaviour of the logarithmic ruin probability.
\end{rem}

If the ellipsoid is not orientated along the axes, it is not possible to find an optimal quota share reinsurance, which is defined component-by-component and results in equal asymptotic behaviour of all directions. Depending on the orientation of the ellipsoid, it might be possible to find a quota share reinsurance strategy defined component-by-component that still reduces the risks in the most risky directions and thus reduces the ruin probability.

%----------------------------------------------------------------------------
\subsection{Quota share risk sharing of the reinsurance company}\label{sec_qsre}
A reinsurance company is interested in optimising the risk sharing portfolio such that 
\begin{displaymath}
\lim_{n\to\infty} \frac{\log(\PP(\|\Sum_n/n\|_2>a))}{h(n)} > \lim_{n\to\infty} \frac{\log(\PP(\|(I-Q)\Sum_n/n\|_2>a))}{h(n)},
\end{displaymath}
where $Q$ is a positive, diagonal $d\times d$ matrix with $0<q_{j,j}<1$ for all $j=1,\dots, d$. The condition $q_{j,j}<1$ for all $j=1,\dots, d$ is due to the natural assumption that the reinsurer offers reinsurance for all lines of business. Comparing the situation with the quota share reinsurance strategy of the ceding company, the same situation leads to a square matrix of smaller dimension since the insurance company want to get an offer for reinsurance only in the most risky lines of business and covers the risks of the line of business with the lowest risk itself. 
The reinsurance company collects the premium 
\begin{displaymath}
p(Q)= {\bf 1}^T (I-Q) \p=\sum_{j=1}^d (1-q_{j,j})p_j
\end{displaymath}
and covers the amount $(1-q_{j,j})X^j$ of the $j$th line of business. The following theorem states the optimising quota share strategy.

\begin{thm}
Let $\X=R\Thetav$ where $\Thetav$ fulfils assumption \ref{as_thetav} and $\Omega=\Lambda(\Ss^{d-1})$ is defined by the linear transformation $\Lambda(\x)=A\x$, where $A$ is a diagonal $d\times d$ matrix  with $a_{k,k}>0$ for all $k=1,\dots,d$. Additionally, we assume that $R$ has a Weibull distribution with parameter $\beta\in (0,1)$.
Then, the quota share reinsurance defined by the matrix $Q=I- A^{-1}/c$ for some $c>\max_{j=1}^d 1/a_{j,j}$ yields the inequality
\begin{equation} \label{eq_asrel2}
\lim_{n\to\infty} \frac{\log(\PP(\|\Sum_n/n\|_2>a))}{h(n)} > \lim_{n\to\infty} \frac{\log(\PP(\|(I-Q)\Sum_n/n\|_2>a))}{h(n)}.
\end{equation}
\end{thm}
\begin{proof}
The large deviations principle implies
\begin{displaymath}
\lim_{n\to\infty} \frac{\log(\PP(\|\Sum_n/n\|_2>a))}{h(n)} =-\inf_{\y\in B(\nolla, a)^c} \|A^{-1}\y\|_2^\beta=-\inf_{\|\y\|_2=a} \|A^{-1}\y\|_2^\beta
\end{displaymath}
and equivalently 
\begin{displaymath}
\lim_{n\to\infty} \frac{\log(\PP(\|(I-Q)\Sum_n/n\|_2>a))}{h(n)} =-\inf_{\|\y\|_2=a} \|A^{-1}(I-Q)^{-1}\y\|_2^\beta.
\end{displaymath}
For the $j$th unit vector $\e_j$ it holds $\|A^{-1}(a\e_j)\|_2=a/a_{j,j}$ which results in
\begin{eqnarray*}
\inf_{\|\y\|_2=a}\|A^{-1}\y\|_2 &\leq& \min\left(\frac{a}{a_{1,1}}, \dots, \frac{a}{a_{d,d}}\right)\leq \max\left(\frac{a}{a_{1,1}}, \dots, \frac{a}{a_{d,d}}\right)\\
&<& ca=\inf_{\|\y\|_2=a}\|A^{-1}(I-Q)^{-1}\y\|_2.
\end{eqnarray*}
This proves Inequality (\ref{eq_asrel2}).
\end{proof}

As in Theorem \ref{thm_reins}, $A$ is a diagonal matrix which implies that the ellipsoid that defines the distribution of $\X$ is orientated along the axes. The constant $c$ defines the risk share of the reinsurance company and hence the amount that they reassure. The quota share ratio of the reinsurer of the $j$th line of business is $1-q_{j,j}=1/(a_{j,j}c)$ so increasing $c$ reduces the risks for the reinsurer. As in Subsection \ref{sec_qsins} also the optimal risk sharing strategy of the reinsurance company includes bigger ratio for the lines of business with smaller risks.

The ceding company as well as the reinsurance company optimise their risks by taking or offering reinsurance that maps their share of the initial ellipsoid to a ball. The optimising strategy of the ceding company yields 
\begin{displaymath}
AQ=\min_{i=1}^d a_{i,i} I
\end{displaymath}
whereas the optimising strategy of the reinsurance company results in
\begin{displaymath}
A(I-\tilde{Q})=\frac{1}{c}I,
\end{displaymath}
where $c>\max_{i=1}^d 1/a_{i,i}$ or $1/c<\min_{i=1}^d a_{i,i}$. 
The radius of the ball defined by $AQ$ is $\sqrt{1/\min_{i=1}^d a_{i,i}}$ and therefore smaller than the radius of the ball defined by $A(I-\tilde{Q})$ which is $\sqrt{c}$. Increasing $c$ increases the radius of the ball generated by $A(I-\tilde{Q})$. 
In general, a bigger radius implies smaller risks for the insurance or reinsurance company. Increasing $c$ reduces the share of the reinsurance company and therefore also the risks of the reinsurance company. 
The minimum $\min_{i=1}^d a_{i,i}$ indicates the line of business with the lowest risk.

%----------------------------------------------------
%   Conclusions
%----------------------------------------------------
\section{Conclusions}

The assumptions of the main result require that the support of the random vectors is, asymptotically, the whole space. In particular, the components of increments are asymptotically dependent. However, the studied model admits more flexibility than many typical models in the sense that it does not require uniformly distributed random vectors on ellipses. This makes it possible to derive asymptotics for a wide class of zero-mean random walks. The general case with the non-zero expectation can be studied by centring the increments.

The derived large deviations principle quantifies, asymptotically, the probabilities of rare events of such random walks, which enables further results in applications such as the presented QS optimisation method. For further research, it would be natural to ask how the set $\Omega$ can be deduced from observed data and if the large deviations principle holds even if $\Omega$ is, for example, any star shaped set.

\bibliographystyle{abbrv}
\bibliography{kirjallisuus} 

\begin{thebibliography}{10}

\bibitem{Bazhba1}
M.~Bazhba, J.~Blanchet, C.-H. Rhee, B.~Zwart, et~al.
\newblock Sample path large deviations for l{\'e}vy processes and random walks
  with weibull increments.
\newblock {\em Annals of Applied Probability}, 30(6):2695--2739, 2020.

\bibitem{Bingham}
N.~Bingham, C.~Goldie, and J.~Teugels.
\newblock {\em Regular Variation}.
\newblock Cambridge University Press, 1989.

\bibitem{Cline1}
D.~B.~H. Cline and S.~I. Resnick.
\newblock Multivariate subexponential distributions.
\newblock {\em Stochastic Process. Appl.}, 42(1):49--72, 1992.

\bibitem{Dembo1}
A.~Dembo and O.~Zeitouni.
\newblock {\em Large deviations techniques and applications}.
\newblock Jones and Bartlett Publishers, Boston, MA, 1993.

\bibitem{Denisov1}
D.~Denisov, A.~B. Dieker, and V.~Shneer.
\newblock Large deviations for random walks under subexponentiality: the
  big-jump domain.
\newblock {\em Ann. Probab.}, 36(5):1946--1991, 2008.

\bibitem{Embrechts1}
P.~Embrechts, C.~Kl\"uppelberg, and T.~Mikosch.
\newblock {\em Modelling extremal events}, volume~33 of {\em Applications of
  Mathematics (New York)}.
\newblock Springer-Verlag, Berlin, 1997.
\newblock For insurance and finance.

\bibitem{Foss1}
S.~Foss, D.~Korshunov, and S.~Zachary.
\newblock {\em An introduction to heavy-tailed and subexponential
  distributions}.
\newblock Springer Series in Operations Research and Financial Engineering.
  Springer, New York, 2011.

\bibitem{Gantert2}
N.~Gantert.
\newblock Functional {E}rdős-{R}enyi laws for semiexponential random
  variables.
\newblock {\em Ann. Probab.}, 26(3):1356--1369, 07 1998.

\bibitem{Gantert1}
N.~Gantert, K.~Ramanan, and F.~Rembart.
\newblock Large deviations for weighted sums of stretched exponential random
  variables.
\newblock {\em Electron. Commun. Probab.}, 19(41):1--14, 2014.

\bibitem{Hagele1}
M.~H\"{a}gele.
\newblock Precise asymptotics of ruin probabilities for a class of multivariate
  heavy-tailed distributions.
\newblock {\em Statist. Probab. Lett.}, 166:108871, 8, 2020.

\bibitem{hardy2001regime}
M.~R. Hardy.
\newblock A regime-switching model of long-term stock returns.
\newblock {\em North American Actuarial Journal}, 5(2):41--53, 2001.

\bibitem{Hult5}
H.~Hult and F.~Lindskog.
\newblock Multivariate extremes, aggregation and dependence in elliptical
  distributions.
\newblock {\em Advances in Applied Probability}, 34(3):587--608, 2002.

\bibitem{Hult3}
H.~Hult and F.~Lindskog.
\newblock On regular variation for infinitely divisible random vectors and
  additive processes.
\newblock {\em Adv. in Appl. Probab.}, 38(1):134--148, 2006.

\bibitem{Hult1}
H.~Hult, F.~Lindskog, T.~Mikosch, and G.~Samorodnitsky.
\newblock Functional large deviations for multivariate regularly varying random
  walks.
\newblock {\em Ann. Appl. Probab.}, 15(4):2651--2680, 2005.

\bibitem{jrfm11030052}
M.~J. Jensen and J.~M. Maheu.
\newblock Risk, return and volatility feedback: A bayesian nonparametric
  analysis.
\newblock {\em Journal of Risk and Financial Management}, 11(3), 2018.

\bibitem{Kluppelberg2}
C.~Kl\"{u}ppelberg, G.~Kuhn, and L.~Peng.
\newblock Estimating the tail dependence function of an elliptical
  distribution.
\newblock {\em Bernoulli}, 13(1):229--251, 2007.

\bibitem{Lehtomaa2}
J.~Lehtomaa.
\newblock Large deviations of means of heavy-tailed random variables with
  finite moments of all orders.
\newblock {\em J. Appl. Probab.}, 54(1):66--81, 2017.

\bibitem{Lehtomaa3}
J.~Lehtomaa and S.~I. Resnick.
\newblock Asymptotic independence and support detection techniques for
  heavy-tailed multivariate data.
\newblock {\em Insurance Math. Econom.}, 93:262--277, 2020.

\bibitem{Mikosch3}
T.~Mikosch and I.~Rodionov.
\newblock Precise large deviations for dependent subexponential variables,
  2020.

\bibitem{Mikosch1}
T.~Mikosch and O.~Wintenberger.
\newblock A large deviations approach to limit theory for heavy-tailed time
  series.
\newblock {\em Probab. Theory Related Fields}, 166(1-2):233--269, 2016.

\bibitem{Nyrhinen1}
H.~Nyrhinen.
\newblock On large deviations of multivariate heavy-tailed random walks.
\newblock {\em J. Theoret. Probab.}, 22(1):1--17, 2009.

\bibitem{Omey1}
E.~Omey.
\newblock Subexponential distribution functions in {${\bf R}^d$}.
\newblock {\em J. Math. Sci. (N.Y.)}, 138(1):5434--5449, 2006.

\bibitem{Samorodnitsky}
G.~Samorodnitsky and J.~Sun.
\newblock Multivariate subexponential distributions and their applications.
\newblock {\em Extremes}, 19(2):171--196, 2016.

\bibitem{Tegner}
M.~Tegnér and R.~Poulsen.
\newblock Volatility is log-normal—but not for the reason you think.
\newblock {\em Risks}, 6(2), 2018.

\bibitem{Teugels}
J.~L. Teugels.
\newblock The class of subexponential distributions.
\newblock {\em Ann. Probability}, 3(6):1000--1011, 1975.

\end{thebibliography}

\end{document}